\documentclass[leqno,a4paper,11pt]{article}
\usepackage{amsmath,amsthm,amssymb}
\usepackage[utf8]{inputenc}
\usepackage[T1]{fontenc}

\usepackage[a4paper]{geometry}
\usepackage{enumerate}
\usepackage{bbm}
\usepackage{mathrsfs}

\usepackage{hyperref}

\usepackage{cleveref}
\crefname{Th}{Theorem}{Theorems}
\crefname{Prop}{Proposition}{Propositions}
\crefname{Lemma}{Lemma}{Lemmas}
\crefname{Cor}{Corollary}{Corollaries}
\crefname{Remark}{Remark}{Remarks}
\crefname{Def}{Definition}{Definitions}
\crefname{Example}{Example}{Examples}
\crefname{Question}{Question}{Questions}
\crefname{section}{Section}{Sections}

\newtheorem{Th}{Theorem}

\newtheorem{Lemma}{Lemma}
\newtheorem{Cor}{Corollary}
\theoremstyle{definition}
\newtheorem{Remark}{Remark}
\newtheorem{Def}{Definition}

\newtheorem{Question}{Question}

\newcommand{\cb}{{\cal B}}
\newcommand{\cc}{{\cal C}}
\newcommand{\cf}{{\cal F}}
\newcommand{\cg}{{\cal G}}
\newcommand{\cp}{{\cal P}}
\newcommand{\cl}{{\cal L}}
\newcommand{\ot}{\otimes}
\newcommand{\ov}{\overline}
\newcommand{\A}{\mathbb{A}}
\newcommand{\Z}{{\mathbb{Z}}}
\newcommand{\N}{{\mathbb{N}}}
\newcommand{\vep}{\varepsilon}
\newcommand{\raz}{\mathbbm{1}}

\newcommand*\samethanks[1][\value{footnote}]{\footnotemark[#1]}

\usepackage{nicefrac}

\title{Hereditary subshifts whose simplex of invariant measures is Poulsen}
\author{Joanna Ku\l{}aga-Przymus\thanks{Research supported by Narodowe Centrum Nauki grant  UMO-2014/15/B/ST1/03736.} \and  Mariusz Lema\'{n}czyk\samethanks \and Benjamin Weiss}

\begin{document}
\maketitle

\begin{abstract}
We give a sufficient condition for the simplex of invariant measures for a hereditary system to be Poulsen. In particular, we show that this simplex is Poulsen in case of positive entropy $\mathscr{B}$-free systems. We also give an example of a positive entropy hereditary system whose simplex of invariant measures is not Poulsen.
\end{abstract}

\section{Introduction}

Recall that given a homeomorphism $T$ of a compact metric space $X$, the set $\mathcal{P}(T,X)$ of probability Borel $T$-invariant measures on $X$ is a Choquet simplex, with the set of extremal points equal to $\cp^e(T,X)$ -- the subset of ergodic measures (see e.g. \cite{MR648108},~\cite{MR1835574} or \cite{MR1958753}, page 95). Recall also that the set of invariant measures of a minimal flow may have the affine-topological structure of an arbitrary metrizable Choquet simplex~\cite{MR1135237}.\footnote{Moreover, all such simplices are obtained for the class of zero entropy Toeplitz subshifts.}

\begin{Def}
We say that a simplex is \emph{Poulsen} if it is non-trivial and the set of its extreme points is dense.
\end{Def}
Recall that every Choquet simplex is affinely homomorphic to a face of a Poulsen simplex~\cite{MR481894}. Moreover, up to affine isomorphism, there is only one Poulsen simplex~\cite{MR500918}. A natural question arises:
\begin{Question}\label{q1}
When is $\cp(T,X)$ a Poulsen simplex?
\end{Question}

We deal with \cref{q1} in case of \emph{hereditary} subshifts, i.e., subshifts $X\subset \{0,1\}^\Z$ with the additional property that for $x,y\in\{0,1\}^\Z$ such that $x\in X$ and $y\leq x$ (coordinatewise), we have $y\in X$ (recall that $X\subset \{0,1\}^\Z$ is called a \emph{subshift} if it is closed and invariant under the left translation $S$).\footnote{In other words, the heredity of $X$ means that this set is closed under coordinatewise multiplication by arbitrary 0-1-sequences.}

Given a subshift $X\subset \{0,1\}^\Z$, we define $\widetilde{X}$ as the smallest hereditary subshift containing $X$, i.e.
$$
\widetilde{X}:=\{y\in\{0,1\}^{\Z}:y\leq x\text{ coordinatewise for some }x\in X\}.
$$
Let $M\colon X\times \{0,1\}^{\Z}\to \{0,1\}^{\Z}$ be given by
$$
M(x,y)(n)=x(n)\cdot y(n)\text{ for each }n\in\Z.
$$
Clearly $M(X\times \{0,1\}^{\Z})=\widetilde{X}$. Suppose now additionally that there exists $\nu\in \mathcal{P}^e(S,X)$ such that
\begin{equation}\label{joiningtype}
\mathcal{P}^e(S,\widetilde{X})=\{M_\ast(\rho) : \rho\in \mathcal{P}^e(S\times S,X\times \{0,1\}^\Z), \rho|_X=\nu \}.
\end{equation}

Our main result is the following:
\begin{Th}\label{tmain}
Suppose that~\eqref{joiningtype} holds. Then $\mathcal{P}(S,\widetilde{X})$ is Poulsen provided that $h_{top}(S,\widetilde{X})>0$.
\end{Th}

Notice also that in the zero entropy case, i.e., when $h_{top}(S,\widetilde{X})=0$, we have  $\mathcal{P}(S,\widetilde{X})=\{\delta_{(\dots,0,0,0,\dots)}\}$ (see~\cite{MR3007694}), in particular, $\mathcal{P}(S,\widetilde{X})$ is not Poulsen.

We apply \cref{tmain} in the following two situations:
\begin{enumerate}[(i)]
\item
for hereditary systems $(S,\widetilde{X})$ with $(S,X)$ being uniquely ergodic,
\item
for $\mathscr{B}$-free systems.
\end{enumerate}
More precisely, in case (i), we have:
\begin{Cor}\label{wn}
Suppose that $(S,X)$ is uniquely ergodic. Then $\mathcal{P}(S,\widetilde{X})$ is Poulsen, provided that $h_{top}(S,\widetilde{X})>0$.
\end{Cor}
\begin{proof}
Given $\mu\in \mathcal{P}^e(S,\widetilde{X})$, let $x$ be a generic point for $\mu$ and let $y\in X$ be such that $x\leq y$. Clearly, $(y,x)\in X\times \{0,1\}^\Z$ is quasi-generic for some measure $\rho$ with $\rho|_X=\nu$. Moreover, $M_\ast(y,x)=x$ and it follows immediately that $M_\ast(\rho)=\mu$. It follows by considering the ergodic decomposition of $\rho$ that~\eqref{joiningtype} holds and we can apply~\cref{tmain}.
\end{proof}
To state the result in case (ii), we need to recall first some basic notions. Let $\mathscr{B}\subset \N$. The set $\cf_{\mathscr{B}}:=\Z\setminus\bigcup_{b\in\mathscr{B}}b\Z$ is called the $\mathscr{B}$-{\em free set}. Let $\eta:=\raz_{\cf_\mathscr{B}}\in \{0,1\}^\Z$ and consider $(S,X_\eta)$ with
$$
X_\eta=\ov{\{S^k\eta:k\in\Z\}}\subset\{0,1\}^{\Z}.
$$
Note that in this case $\widetilde{X}_\eta\subset X_{\mathscr{B}}$, where $(S,X_{\mathscr{B}})$ is the $\mathscr{B}$-{\em admissible subshift}~\cite{Ba-Ka-Ku-Le}:
$z\in X_{\mathscr{B}}$ if and only if we have $|\{n\in\Z: z(n)=1\}\text{ mod }b|<b$ for each $b\in \mathscr{B}$. There is a natural $S$-invariant measure $\nu_\eta$ on $X_\eta$, called the Mirsky measure, for which $\eta$ is quasi-generic~\cite{Ba-Ka-Ku-Le}. This means that there exists an increasing sequence $(N_k)_{k\geq 1}\subset \N$ such that
$$
\frac{1}{N_k}\sum_{n\leq N_k}f(S^n\eta)=\int f\ d\nu_\eta
$$
for any continuous function $f$ on $X_\eta$. Moreover, $\nu_\eta$ is of zero entropy and we have:
\begin{Th}[\cite{Ba-Ka-Ku-Le}, for $\mathscr{B}$ pairwise coprime with $\sum_{b\in\mathscr{B}}1/b<\infty$, see~\cite{MR3356811}]\label{t:bkkl1}
For any $\mathscr{B}\subset\N$,
$$
\mathcal{P}^e(S,\widetilde{X}_\eta)=\{M_\ast(\rho) : \rho \in \cp^e(S\times S, X_\eta\times\{0,1\}^{\Z}), \rho|_{X_\eta}=\nu_\eta\}.
$$
\end{Th}
We also have:
\begin{Th}[\cite{Ba-Ka-Ku-Le}]\label{entropia}
For any $\mathscr{B}\subset \N$, $h_{top}(S,\widetilde{X}_\eta)=h_{top}(S,X_\mathscr{B})=\ov{d}(\cf_\mathscr{B})$.\footnote{For $A\subset \Z$, $\ov{d}(A)$ stands for the upper density of $A\cap \N$.}
\end{Th}

As an immediate consequence of \cref{tmain}, \cref{t:bkkl1} and \cref{entropia}, we obtain:
\begin{Cor}\label{c2}
Let $\mathscr{B}\subset\N$ be such that $\ov{d}(\cf_\mathscr{B})>0$. Then $\mathcal{P}(S,\widetilde{X}_\eta)$ is Poulsen.
\end{Cor}
In many cases, the $\mathscr{B}$-free system $(S,X_\eta)$ itself is hereditary, i.e.\ $\widetilde{X}_\eta=X_\eta$. In particular, using results from \cite{Ba-Ka-Ku-Le}, we obtain the following:
\begin{Cor}\label{c:main1}
Let $\mathscr{B}\subset\N$ be such that $\ov{d}(\cf_{\mathscr{B}})>0$, with $\ov{d}(\bigcup_{b\geq K}b\Z)\to 0$ when $K\to\infty$. Assume moreover that $\mathscr{B}$ contains an infinite pairwise coprime subset. Then $(S,X_\eta)$ is hereditary and the simplex $\cp(S,X_\eta)$ is Poulsen.
\end{Cor}
In particular, $\cp(S,X_\eta)$ is Poulsen in the following classical cases:
\begin{itemize}
\item when $\mathscr{B}$ is infinite, pairwise coprime and $\sum_{b\in \mathscr{B}}1/b<\infty$; for example, the result holds for the square-free subshift given by $\mathscr{B}=\{p^2:p\text{ is prime}\}$;
\item when $\mathscr{B}=\mathscr{B}_{\A}$, where  $\mathscr{B}_{\A}$ stands for the set of primitive abundant numbers~\cite{MR1574879}.
\end{itemize}

Finally, we give an example of a hereditary system of positive entropy whose simplex of invariant measures fails to be Poulsen.

\begin{Remark}
For $\mathscr{B}$ pairwise coprime with $\sum_{b\in \mathscr{B}}1/b<\infty$ another proof of~\cref{c2} has been presented in~\cite{ko-kwi}. This proof is using a different method than ours.
\end{Remark}

\begin{Remark}
Our original motivation for \cref{tmain} was to study the simplex of invariant measures for $\mathscr{B}$-free systems, where the Mirsky measure $\nu_\eta$ that plays the role of $\nu$ in condition~\eqref{joiningtype} is of zero entropy. For this reason, we include a complete proof of~\cref{tmain} under the extra assumption that $\nu$ has zero entropy and then explain the necessary changes to obtain the full version of our result.
\end{Remark}

\begin{Remark}
While the name \emph{Poulsen simplex} comes from~\cite{MR0123903}, where a simplex with a dense set of extreme points was constructed, this is historically not the first such example. The most basic dynamical system with the simplex of invariant measures being Poulsen is the full shift. For the 2-shift, the fact that any invariant measure can be approximated by measures concentrated on periodic orbits (such measures are of course ergodic) follows from~\cite{Ville1939} (see the proof of Theorem 3 therein, in particular, the comments on page 13). For the full shift over any Polish space, the fact that ergodic measures are dense in the space of all invariant measures was proved by Parthasarathy~\cite{MR0148850}. There is a further discussion of this with a proof that measures concentrated on periodic orbits are dense in the paper of Oxtoby \cite{MR0160875}. Moreover, Sigmund~\cite{MR0286135, MR0352411} gave a condition on $(T,X)$ (so-called periodic specification property) that implies that $\cp(T,X)$ is Poulsen. His results were applied in many situations, e.g., in~\cite{MR0457675,MR1765882,MR718829,MR1407484}. See~\cite{Gelfert:2014aa} for more details.
\end{Remark}

\section{Proof of \cref{tmain} for $\nu$ of zero entropy}\label{se2}

\begin{Lemma}\label{l:bm1}
Assume that $(X_i,\cb_i,\mu_i)$, $i=1,2$, are standard probability Borel spaces with automorphisms $T_i$. Assume that $A\subset X_1\times X_2$ is a Borel set and let $\rho\in J(T_1,T_2)$.\footnote{$\rho\in J(T_1,T_2)$ is a {\em joining} of $T_1$ and $T_2$, i.e.\ $\rho\in \mathcal{P}(T_1\times T_2,X_1\times X_2)$ and has projections $\mu_1$ and $\mu_2$ respectively.} Let $\pi_{X_1}\colon X_1\times X_2\to X_1$ stand for the projection onto the first coordinate. Then $\mu_1(\pi_{X_1}(A))\geq \rho(A)$.
\end{Lemma}

\begin{proof} Since $\pi_{X_1}$ is Borel measurable and $A$ is Borel, the set $\pi_{X_1}(A)$ is $\mu_1$-measurable. Moreover, $
A\subset \pi_{X_1}(A)\times X_2$ and the result follows.
\end{proof}

Given a subshift $Y\subset\{0,1\}^{\Z}$, let $\cl(Y)$ stand for the family of all blocks appearing in a $y\in Y$.

\begin{Lemma}\label{l:bm2}
Given $n\geq1$, $\delta>0$, $\mu\in\cp^e(S,\widetilde{X})$, and $A_n\subset\{0,1\}^n$ with $\mu(A_n)>1-\delta$,
let $$\cc_n:=\{u\in\cl(X):|u|=n\text{ and }u\geq w\text{ for some }w\in A_n\}.$$
Then $\nu(\cc_n)>1-\delta$, where $\nu$ is as in~\eqref{joiningtype}.
%Then, for the Mirsky measure $\nu_\eta$ on $X_\eta$, we have $\nu_\eta(\cc_n)>1-\delta$.
\end{Lemma}
\begin{proof}
Let $\rho\in \mathcal{P}(S\times S, X\times \{0,1\}^\Z)$ be such that $M_\ast(\rho)=\mu$ and $\rho|_X=\nu$. It is not hard to see $C_n\supset \pi_{X}(M^{-1}(A_n))$. The result follows from \cref{l:bm1}.
\end{proof}

Fix $\nu_1,\nu_2\in \cp^e(S,\widetilde{X})$. All we need to show is that the measure $\frac12(\nu_1+\nu_2)$ can be approximated by ergodic measures.\footnote{It is enough to show that  the closure of the
ergodic measures is a convex set and for this it
suffices to verify the midpoint condition for ergodic measures.}
By the definition of weak topology on measures, it follows that, given $k_0\geq1,\vep_0>0$, we need
to find $\ov{\eta}\in \widetilde{X}$ such that:
\begin{equation}\label{b1}
\parbox{0.8\textwidth}{
$\ov{\eta}$ is generic for an ergodic measure,
}
\end{equation}
\begin{equation}\label{b2}
\parbox{0.8\textwidth}{
the empirical distribution of $k_0$-blocks on $\ov{\eta}$ is, up to $\vep_0>0$, equal to $\frac12(\nu_1|_{\{0,1\}^{k_0}}+\nu_2|_{\{0,1\}^{k_0}})$.
}
\end{equation}

Fix $k_0\geq1$ and $\vep_0>0$. Fix also $\vep>0$ much smaller than $\vep_0$.
Using the ergodic theorem for $\nu_1$ and $\nu_2$ respectively, we can find $n_0\geq k_0$, $\cf_i\subset\{0,1\}^{n_0}$, $\nu_i(\cf_i)>1-\vep/2$,  $i=1,2$ such that
\begin{equation}\label{b1a}
\parbox{0.8\textwidth}{
the empirical $k_0$-distribution in any $w\in\cf_i$ is $\vep$-close to $\nu_i|_{\{0,1\}^{k_0}}$.
}
\end{equation}
Let
\begin{equation}\label{b12}
\cg_{n_0}:=\{u\in\cl(X_\eta): |u|=n_0\text{ and }u\geq w_i\text{ for some } w_i\in \cf_i, i=1,2\}.
\end{equation}
Note that, for $i=1,2$, we have
$$
u\in \cl(X) \text{ and } u\geq w_i\text{ for some } w_i\in\cf_i\iff u\in \pi_X(M^{-1}(\cf_i)),
$$
i.e.\
\begin{equation}\label{w2}
\cg_{n_0}=\pi_X(M^{-1}(\cf_1)) \cap \pi_X(M^{-1}(\cf_2)).
\end{equation}
Therefore, applying \cref{l:bm2} to $n=n_0$ and $A_n=\cf_i$, $i=1,2$, we obtain
$$
\nu(\cg_{n_0})>1-\vep.
$$
Based on~\eqref{b12}, we define two maps $R_i\colon\cg_{n_0}\to\{0,1\}^{n_0}$, so that
\begin{equation}\label{b112}
R_i(u)\in \cf_i,\;i=1,2.
\end{equation}

Consider now the Markov chain with the states $\{1,\ldots,n_0,n_0+1\}$ and the transition probabilities, i.e.\ the stochastic matrix $P'=(p'_{ij})$, given by:
$$
p'_{i,i+1}=1\text{ for  }i=1,\ldots,n_0-1,\ p'_{n_0,n_0+1}=1/2,\ p'_{n_0,1}=1/2,\ p'_{n_0+1,1}=1.$$
Let $p'=(p'_1,\ldots,p'_{n_0},p'_{n_0+1})$ be the probabilistic vector in which
$$
p'_i=\frac1{n_0+1/2}\text{ for }i=1,\ldots,n_0\text{ and }p'_{n_0+1}=\frac1{2(n_0+1/2)}.$$
We have $p'\cdot P'=p'$ whence the formula
$$
\kappa'(a_0,a_1,\ldots,a_m)=p'_{a_0}p'_{a_0,a_1}\ldots p'_{a_{m-1},a_m}$$
for $a_k\in\{1,\ldots,n_0,n_0+1\}$ yields an $S$-invariant (Markov) measure $\kappa'$ on $W':=\{1,\ldots,n_0,n_0+1\}^{\Z}$.
It is not hard to see that $(P')^{n_0+1}$ has all entries positive, that is, $P'$ is aperiodic, and therefore the Markov measure $\kappa'$ yields a mixing Markov shift. 

Let $\eta_\nu\in \{0,1\}^\Z$ be a generic point for $\nu$.
\begin{Lemma}\label{l:bm4} If $z'\in\{0,1\}^{\Z}$ is a generic point for $\kappa'$, then $(\eta_\nu,z')$ is a generic point for the product measure $\nu\ot\kappa'$ with the latter measure being ergodic.
\end{Lemma} \begin{proof}Each mixing Markov subshift is a K-system, so the result follows directly from disjointness of zero entropy systems with K-systems, see, e.g.,~\cite{MR1958753}.\end{proof}

Notice that $z'$ above consists of consecutive blocks
$$(1,2,\ldots,n_0)\text{ or }(1,2,\ldots,n_0,n_0+1).$$ Moreover, by \cref{l:bm4},
\begin{equation}\label{b14}
(1-\vep)\kappa'([1])\leq\nu\ot\kappa'(\cg_{n_0}\times [1])=\lim_{N\to\infty}\frac1N\sum_{s\leq N}\raz_{\cg_{n_0}\times[1]}((S\times S)^s(\eta_\nu,z')).
\end{equation}

\begin{Remark} The natural representation of $z'$ as a concatenation of blocks of length $n_0$ and $n_0+1$ induces the corresponding representation of $\eta_\nu$ as concatenation of block of the same lengths.   An interpretation of~\eqref{b14} is that when we look at the concatenation of $\eta_\nu$ as $n_0$- and $(n_0+1)$-blocks, then for ``most'' of  the blocks, we see that the either the $n_0$-block belongs to $\cg_{n_0}$ or the beginning $n_0$-block in case of length $n_0+1$, belongs to $\cg_{n_0}$. Note also that we cannot simplify this argument by representing $\eta_\nu$ as concatenation only of $n_0$-blocks. Indeed, we do not  know whether $(S^{n_0},X,\nu$) is ergodic, hence we cannot be sure that for most of the blocks in such a concatenation, we are in $\cg_{n_0}$. For example, if we consider the square-free case, $(S^{m},X_\eta,\nu_{\eta})$ is not ergodic for any $m\geq2$ as the spectrum contains $e^{2\pi i/p^2}$, $p\in \mathscr{P}$, hence the roots of all prime degrees.\end{Remark}

Remembering that due to the natural representation of $z$ as concatenation of $n_0$- and $(n_0+1)$-blocks, the sequence $\eta_\nu$ is represented as a concatenation of blocks of length $n_0$ or $n_0+1$, let us now define a new sequence $\eta'=\eta'(\eta_\nu,z)\in\{\ast,0\}^{\Z}$ in the following way:
\begin{enumerate}[(a)]
\item
if in the above concatenation the block in $\eta_\nu$ is of length $n_0$ and belongs to $\cg_{n_0}$, we replace it by the all $\ast$ $n_0$-block;
\item
if in the above concatenation the block in $\eta_\nu$ is of length $n_0+1$ and the starting $n_0$-block is in $\cg_{n_0}$, we replace it by the all $\ast$ $n_0$-block adding 0 at the end to obtain a block of length $n_0+1$;
\item
if in the above concatenation the block in $\eta_\nu$ is of length $n_0$ and does not belong to $\cg_{n_0}$ or it is of length $n_0+1$ but the starting $n_0$-block neither belongs to $\cg_{n_0}$, we replace it by the all $0$ $n_0$- or $(n_0+1)$-block.
\end{enumerate}

\begin{Lemma}\label{l:b16}
The sequence $\eta'\in\{\ast,0\}^{\Z}$ is generic for an ergodic measure.
\end{Lemma}
\begin{proof} We simply show that $\eta'$ is obtained from $(\eta_\nu,z)$ by a (finite) code.\footnote{Formally, we should define a block map. This can be done for example as follows. Consider all blocks (over the double alphabet) of length $4n_0+1$ that appear in $(\eta_\nu,z')$. Due to the special form of $z'$, there will be at least 3 symbols 1 on the second coordinates of this block. Look for the first 1 on the left of the middle of the block. This gives a certain position $j$. Look at the block on $\eta_\nu$ of length $n_0$ (or $n_0+1$ if , as the symbol(!), $n_0+1$ appears (on the second coordinate) on the right to the middle position before 1 reappears), read whether the corresponding block on $\eta_\nu$ belongs or does not to $\cg_{n_0}$ and code the whole $4n_0+1$-block (over the double alphabet) by $\ast$ or $0$, respectively.} Indeed, if we want to determine $\eta'(i)$ (i.e. $\ast$ or 0), we first look at $z(i)$ and seek the first symbol 1 on the left (we do not check more than $n_0+1$ positions); we determined a position $j$ in this way, and we check now whether $\eta_\nu(j,j+n_0-1)$ does or does not belong to $\cg_{n_0}$. Now the $\eta'(i)$ is determined by (a)-(c).\end{proof}

Consider now the following Markov shift: the set of states consists of $$\{1,\ldots,n_0,n_0+1\}\text{ and its disjoint copy } \{\ov{1},\ldots,\ov{n_0},\ov{n_0+1}\}.$$ Let
$$
p=(p_1,\ldots,p_{n_0},p_{n_0+1}, p_{\ov{1}},\ldots,p_{\ov{n_0}},p_{\ov{n_0+1}})$$
be the probabilistic vector for which $p_i=p_{\ov{i}}=\frac{1}{2n_0+1}$ for $i=1,\ldots, n_0$ and $p_{n_0+1}=p_{\ov{n_0+1}}=\frac1{2(2n_0+1)}$.
The matrix $P$ of transition probabilities is given by the following:
$$
p_{i,i+1}=p_{\ov{i},\ov{i+1}}=1\text{ for  }i=1,\ldots,n_0-1,$$
$$p_{n_0,n_0+1}=p_{\ov{n_0},\ov{n_0+1}}=1/2,$$
$$
p_{n_0,1}=p_{\ov{n_0},\ov{1}}=p_{n_0,\ov{1}}=p_{\ov{n_0},1}=1/4$$
and
$$ p_{n_0+1,1}=p_{n_0+1,\ov{1}}=
p_{\ov{n_0+1},1}=p_{\ov{n_0+1},\ov{1}}=1/2.$$
We have $p\cdot P=p$ whence the formula
$$
\kappa(a_0,a_1,\ldots,a_m)=p_{a_0}p_{a_0,a_1}\ldots p_{a_{m-1},a_m}$$
for $a_k\in\{i,\ov{i}: i=1,2,\ldots,n_0\}$ yields an $S$-invariant (Markov) measure $\kappa$ on $W:=\{1,\ldots,n_0+1,\ov{1},\ldots,\ov{n_0+1}\}^{\Z}$.
It is not hard to see that $P^{2(n_0+1)}$ has all entries positive, that is, $P$ is aperiodic, and therefore the Markov measure $\kappa$ yields a mixing Markov shift. 

Fix $z$ a generic point for $\kappa$. Then $z$ is a concatenation of 4 types of blocks (of length either $n_0$ or $n_{0}+1$): $(1,\ldots,n_0)$, $(1,\ldots,n_0,n_0+1)$, $(\ov{1},\ldots,\ov{n_0})$ and $(\ov{1},\ldots,\ov{n_0},\ov{n_0+1})$. Moreover, on a sufficiently long initial part of $z$, the non-barred-blocks and the barred-blocks are equally probable (as $\kappa(1)=\kappa(\ov{1})$). Similarly, to \cref{l:b16}, we obtain the following.

\begin{Lemma}\label{l:bm6} If $z\in\{0,1\}^{\Z}$ is a generic point for $\kappa$, then $(\eta_\nu,z)$ is a generic point for the product measure $\nu\ot\kappa$ with the latter measure being ergodic.
\end{Lemma}

Moreover, the map  which to $i$ and $\ov{i}$ associates $i$, for $i=1,\ldots,n_0,n_0+1$ yields a continuous factor map $\Lambda$ between
$(S,W,\kappa)$ and $(S,W',\kappa')$; in particular, if $z$ is generic for $\kappa$ then $z':=\Lambda(z)$ is generic for $\kappa'$.

We now fix $z$ a generic point for $\kappa$ and  repeat the construction of $\eta'\in\{0,1\}^{\Z}$ for $z'$. Then, we transform $\eta'$ into $\ov{\eta}$ by replacing the $n_0$-$\ast$-block $w$ by either $R_1(w)$ or $R_2(w)$ (cf.~\eqref{b112}) depending on the fact whether on $z$ we considered the non-barred- or the barred-block. Up to density $\epsilon_1>0$, we have now replaced half of the $\ast$-blocks by blocks from $\cf_1$ and the second half by blocks from $\cf_2$. Clearly, $\ov{\eta}\leq\eta_\nu$. Moreover, by~\eqref{b1a}, the empirical distribution of $k_0$-blocks on $\ov{\eta}$ is $\vep$-close to the distribution of $k_0$-blocks for the measure $\frac12(\nu_1+\nu_2)$. Finally, following the proof of \cref{l:b16}, we obtain that $\ov{\eta}$ is generic for an ergodic measure (it is obtained by a finite code from $(\eta_\nu,z)$). We have proved~\eqref{b1} and~\eqref{b2}, so \cref{tmain} follows.

\section{Sketch of the proof of \cref{tmain} for general $\nu$}

We explain now how to modify the proof from \cref{se2} after dropping the additional assumption that $\nu$ has zero entropy. Instead of the Markov shift $(S,W,\kappa)$, we consider an arbitrary ergodic aperiodic automorphism $(T,{W},\kappa)$ disjoint from $(S,X,\nu)$ (the existence of such an automorphism follows from \cite{MR633762}, in fact, a generic automorphism is disjoint from $(S,X,\nu)$). By the Alpern's Lemma~\cite{MR661814}, $W$ can be decomposed into two towers, say, $\mathcal{W}_0$ and $\mathcal{W}_1$, of height $n_0$ and $n_0+1$, respectively. We split each of them into two further towers of equal measure and of the same height as the original tower: $\mathcal{W}_0=\mathcal{W}_0^L\cup\mathcal{W}_0^R$, $\mathcal{W}_1=\mathcal{W}_1^L\cup\mathcal{W}_1^R$. By assigning symbols $1,\dots, n_0$ and $1,\dots, n_0+1$ to the consecutive levels of $\mathcal{W}_0^L$ and $\mathcal{W}_1^L$, and symbols $\ov{1},\dots, \ov{n_0}$ and $\ov{1},\dots, \ov{n_0+1}$ to the consecutive levels of $\mathcal{W}_0^L$ and $\mathcal{W}_1^L$, we obtain a coding of points from $W$ by two-sided sequences over the alphabet $\{1,\dots,n_0,n_0+1,\ov{1},\dots, \ov{n_0},\ov{n_0+1}\}$. The remaining part of the proof of \cref{tmain} stays the same as in \cref{se2}.

\section{Hereditary system of positive entropy whose simplex of invariant measures is not Poulsen}

In this section we will show that there are hereditary systems of positive entropy whose simplex of invariant measures is not Poulsen. For this, we recall an example, which was used in~\cite{MR3356811} to show that there are hereditary systems which are not intrinsically ergodic.

Given a block $C\in \{0,1\}^n$, let $x_C\in \{0,1\}^\Z$ be the infinite concatenation of $C$ and let $X_C\subset \{0,1\}^\Z$ stand for the orbit closure of $x_C$ (equal to its orbit). Let $\widetilde{X}_C$ be the smallest hereditary subshift containing $X_C$. Finally, let $\nu_C$ be the periodic measure for which $x_C$ is a generic point. Notice that if $\text{supp }C \neq \emptyset$ then $\nu_C\neq \delta_{(\dots,0,0,0,\dots)}$.

Let $A:=101001000$, $B:=101000100$ and consider $X:=X_A\cup X_B$. Then the smallest hereditary subshift $\widetilde{X}$ containing $X$ equals $\widetilde{X}_A\cup \widetilde{X}_B$. Since both $(S,X_A)$ and $(S,X_B)$ are uniquely ergodic (with zero entropy), it follows by \cref{wn} that both $\mathcal{P}(S,\widetilde{X}_A)$ and $\mathcal{P}(S,\widetilde{X}_B)$ are Poulsen simplices (they are both non-trivial as $\nu_A,\nu_B\neq \delta_{(\dots,0,0,0,\dots)}$).

Suppose that $\mathcal{P}(S,\widetilde{X})$ is also Poulsen and take its arbitrary element $\nu$. Since $\nu$ can be approximated by ergodic measures and $\mathcal{P}^e(S,\widetilde{X})=\mathcal{P}^e(S,\widetilde{X}_A)\cup \mathcal{P}^e(S,\widetilde{X}_B)$, it follows that
\begin{equation}\label{poul}
\nu\in \mathcal{P}(S,\widetilde{X}_A)\cup \mathcal{P}(S,\widetilde{X}_B).
\end{equation}
Let now $\nu:=\frac{1}{2}(\nu_A+\nu_B)$. Since $\nu_A(X_A)=\nu_B(X_B)=1$ and $X_A\cap X_B=\emptyset$, it follows that $\nu(\widetilde{X}_A)=\nu(\widetilde{X}_B)=1/2$. This contradicts~\eqref{poul} and we conclude that $\mathcal{P}(S,\widetilde{X})$ cannot be Poulsen.

Finally, notice that it is easy to modify the above example, so that $(S,\widetilde{X})$ becomes intrinsically ergodic (i.e.\ it has exactly one measure of maximal entropy). E.g., we can take $A':=111001000$ instead of $A$ and consider $X':=X_{A'}\cup X_B$. Then $M_\ast(\nu_{A'}\otimes B(1/2,1/2))$ (where $B(1/2,1/2)$ stands for the Bernoulli measure $(1/2,1/2)$ on $\{0,1\}^\Z$) is the unique measure of maximal entropy for $(S,\widetilde{X'})$, see~\cite{MR3356811} for more details.

\vspace{2ex}

\noindent Joanna Ku\l aga-Przymus:\\
Institute of Mathematics, Polish Academy of Sciences, \'{S}niadeckich 8, 00-956 Warsaw, Poland \\
and
\\
Faculty of Mathematics and Computer Science, Nicolaus Copernicus University, Chopina 12/18, 87-100 Toru\'{n}, Poland\\
\textit{E-mail address:} \texttt{joanna.kulaga@gmail.com}\\
\\
\noindent Mariusz Lema\'nczyk:\\
Faculty of Mathematics and Computer Science, Nicolaus Copernicus University, Chopina 12/18, 87-100 Toru\'{n}, Poland\\
\textit{E-mail address:} \texttt{mlem@mat.umk.pl}\\
\\
\noindent
Benjamin Weiss:\\
Institute of Mathematics, Hebrew University of Jerusalem, Jerusalem,
Israel\\
\textit{E-mail address:} \texttt{weiss@math.huji.ac.il}

\end{document}